\documentclass[a4paper, 10pt]{amsart}
\usepackage{amsmath}
\usepackage{amssymb}
\usepackage{hyperref}

\theoremstyle{plain}
\newtheorem{theorem}{\bf Theorem}[section]
\newtheorem{proposition}[theorem]{\bf Proposition}
\newtheorem{lemma}[theorem]{\bf Lemma}

\theoremstyle{definition}

\newtheorem{definition}[theorem]{\bf Definition}

\newcommand{\N}{\mathbb N}

 \DeclareMathOperator{\ord}{ord}

 \DeclareMathOperator{\supp}{supp}

\renewcommand{\t}{\, | \,}

\numberwithin{equation}{section}

\begin{document}

\title{Products of $k$ atoms in Krull monoids}

\author{Yushuang Fan }
\address{Mathematical College, China University of Geosciences (Beijing)\\
Haidian District, Beijing, the People's Republic of China}
\email{fys@cugb.edu.cn
}

\author{Qinghai Zhong}

\address{Institute for Mathematics and Scientific Computing,  University of Graz, NAWI Graz \\
 Heinrichstra{\ss}e 36, 8010 Graz, Austria}
 \email{qinghai.zhong@uni-graz.at}

\thanks{Main part of this manuscript was written while the first author visited University of Graz. She would like to gratefully acknowledge the kind hospitality from the host institute and  the authors would like to thank Professor Alfred Geroldinger  for his many helpful suggestions.  This work was supported by the Austrian Science Fund FWF (Project No. M1641-N26),
  NSFC (grant no. 11401542), and the Fundamental Research Funds for the Central Universities (grant no. 2652014033).}



\begin{abstract}
Let $H$ be a Krull monoid with finite class group $G$ such that every class contains a prime divisor. For $k\in \N$, let $\mathcal U_k(H)$ denote the set of all $m\in \N$ with the following property: There exist atoms $u_1, \ldots, u_k, v_1, \ldots , v_m\in H$ such that $u_1\cdot\ldots\cdot u_k=v_1\cdot\ldots\cdot v_m$. It is well-known that the sets $\mathcal U_k (H)$ are finite intervals whose maxima $\rho_k(H)=\max \mathcal U_k(H) $ depend only on $G$.  If $|G|\le 2$, then $\rho_k (H) = k$ for every $k \in \N$. Suppose that $|G| \ge 3$. An elementary counting argument shows that  $\rho_{2k}(H)=k\mathsf D(G)$ and  $k\mathsf D(G)+1\le \rho_{2k+1}(H)\le k\mathsf D(G)+\lfloor \frac{\mathsf D(G)}{2}\rfloor$ where  $\mathsf D(G)$ is the Davenport constant. In \cite{Ga-Ge09b} it was proved that for cyclic groups we have $k\mathsf D(G)+1 = \rho_{2k+1}(H)$ for every $k \in \N$. In the present paper we show that (under a mild condition on the Davenport constant) for every noncyclic group there exists a $k^*\in \N$ such that $\rho_{2k+1}(H)= k\mathsf D(G)+\lfloor \frac{\mathsf D(G)}{2}\rfloor$ for every $k\ge k^*$.
This confirms a conjecture of  A. Geroldinger, D. Grynkiewicz, and P. Yuan in \cite{Ge-Gr-Yu15}.
\end{abstract}

\maketitle

\section{Introduction}

Let $H$ be an atomic monoid. If an element $a \in H$ has a factorization $a = u_1 \cdot \ldots \cdot u_k$  into atoms $u_1, \ldots, u_k \in H$, then $k$ is called the length of the factorization, and the set $\mathsf L (a)$ of all possible lengths is called the set of lengths of $a$. For $k \in \mathbb N$, let $\mathcal U_k (H)$ denote the set of all $m \in \mathbb N$ with the following property: There exist atoms $u_1, \ldots, u_k, v_1, \ldots, v_m \in H$ such that $u_1 \cdot \ldots \cdot u_k = v_1 \cdot \ldots \cdot v_m$. Thus $\mathcal U_k (H)$ is the union of all sets of lengths containing $k$. The sets $\mathcal U_k (H)$ are one of the most investigated invariants in factorization theory which were  introduced by S.T. Chapman and W.W. Smith in  Dedekind domains (\cite{Ch-Sm90a}). Their suprema $\rho_k (H) = \sup \mathcal U_k (H)$ were first studied in the 1980s for rings of integers in algebraic number fields (\cite{Cz81, Sa-Za82}). Since then these invariants have been studied  in a variety of settings, including numerical monoids, monoids of modules, noetherian and Krull domains (for a sample out of many we refer to
\cite{Fr-Ge08, Bl-Ga-Ge11a, B-C-H-P08, Ge-Ka-Re15a, Ba-Ge14b}).

In the present paper we focus on Krull monoids with class group $G$ such that every class contains a prime divisor. If $|G|\le 2$, then $\mathcal U_k (H) = \{k\}$ and if $G$ is infinite, then $\mathcal U_k (H) = \N_{\ge 2}$ for all $k \in \N$. Suppose that $G$ is finite with $|G| \ge 3$. This setting includes holomorphy rings in global fields. For more examples we refer to \cite{Ge-Gr-Yu15}, and a detailed exposition of Krull monoids can be found in \cite{HK98, Ge-HK06a}.

The unions
$\mathcal U_k (H) \subset \mathbb N$ are finite intervals, say $\mathcal U_k (H) = [\lambda_k (H), \rho_k (H)]$, whose minima $\lambda_k (H)$ can be expressed in terms of $\rho_k (H)$ (\cite[Chapter 3]{Ge09a}). Elementary counting arguments (e.g. \cite[Section 6.3]{Ge-HK06a}) show that, for every $k \in \N$, we have $\rho_{2k} (H) = k \mathsf D (G)$ and that
\begin{equation} \label{inequality}
k\mathsf D(G)+1\le \rho_{2k+1}(H)\le k\mathsf D(G)+\left\lfloor \frac{\mathsf D(G)}{2}\right\rfloor \,.
\end{equation}
Based on the Savchev-Chen Structure Theorem \cite[Section 11.3]{Gr13a} (resp. on a related result on the index of sequences)
Gao and Geroldinger \cite{Ga-Ge09b} showed that for every cyclic group  $G$ and every $k\in \N$ we have $\rho_{2k+1}(H)= k\mathsf D(G)+1$.
In \cite[Conjecture 3.3]{Ge-Gr-Yu15}, the authors conjectured that for every noncyclic group $G$ there exists a $k^* \in \N$ such that
\[
\rho_{2k+1}(H) = k\mathsf D(G)+\left\lfloor \frac{\mathsf D(G)}{2}\right\rfloor \quad \text{for every} \quad k \ge k^* \,.
\]
We confirm this conjecture for wide classes of groups. For a precise formulation of our main result we need one more definition.
Suppose that  $G \cong C_{n_1}\oplus \ldots \oplus C_{n_r}$ where $r, n_1,\ldots, n_r \in \N$ with $1<n_1\t  \ldots \t n_r$, and set
\[
\mathsf D^*(G)=1+\sum_{i=1}^r (n_i-1) \,.
\]
It is well-known that  $\mathsf D^* (G) \le  \mathsf D(G)$. Equality holds for  $p$-groups, groups of rank at most two, and others (see \cite[Corollary 4.2.13]{Ge09a}, \cite{Bh-SP07a}  for recent progress), but it does not hold in general (\cite{Ge-Li-Ph12}). Here is our main result.

\begin{theorem} \label{Main}
Let $H$ be a Krull monoid with finite noncyclic class group $G$ such that every class contains a prime divisor.
Then there exists a $k^*\in \N$ such that
\[
\rho_{2k+1}(H)\ge (k-k^*)\mathsf D(G)+ k^*\mathsf D^*(G)+\left\lfloor \frac{\mathsf D^*(G)}{2}\right\rfloor \qquad \text{ for every} \ k\ge k^*\,.
\]
In particular, if $\mathsf D(G)=\mathsf D^*(G)$, then $$\rho_{2k+1}(H)= k\mathsf D(G)+\left\lfloor \frac{\mathsf D(G)}{2}\right\rfloor \qquad \text{ for every} \ k\ge k^*\,.$$
\end{theorem}

In \cite{Ge-Gr-Yu15},  Geroldinger, Grynkiewicz, and  Yuan gave a list of  groups for which the above result holds with $k^* =1$. Furthermore, they showed that if $G \cong C_m \oplus C_{mn}$ with $n \ge 1$ and $m \ge 2$, then the result holds with $k^*=1$ if and only if $n=1$ or $m=n=2$. It remains a challenging task to determine, for a given group $G$,  the smallest possible $k^* \in \N$ for which the above statement holds.

It is well-known that the invariants $\rho_k (H)$ can be studied in an associated monoid of zero-sum sequences and this allows to use methods from Additive Combinatorics (see Lemma \ref{2.1}). In Section \ref{2} we fix our notation and terminology. At the beginning of Section \ref{3} we introduce our main concept in Definition \ref{def-nice} and after that we discuss the strategy of the proof.

\bigskip
\section{Preliminaries} \label{2}

Let $\mathbb N$ denote the set of positive integers and $\mathbb{N}_0=\mathbb{N}\cup\{0\}$. For  real numbers $a, b \in \mathbb R$, we denote by  $[a, b] = \{ x \in
\mathbb Z \mid a \le x \le b\}$ the discrete interval. For $n \in \N$ we denote by $C_n$ a cyclic group of order $n$. Let $G$ be a finite abelian group. Then $G \cong C_{n_1} \oplus \ldots \oplus C_{n_r}$ where $r \in \N_0$, $n_1, \ldots, n_r \in \N$ with $1 < n_1 \t \ldots \t n_r$. We call $r = \mathsf r (G)$ the {\it rank} of $G$ (thus $\mathsf r (G)$ is the maximum of the $p$-ranks of $G$), and a tuple $(e_1, \ldots, e_s)$ of nonzero elements of $G$ is said to be a {\it basis} of $G$ if $G = \langle e_1 \rangle \oplus \ldots \oplus \langle e_s \rangle$.
We start with a couple of remarks on abstract monoids, continue with the monoid of zero-sum sequences, and then we deal  with Krull monoids.

By a   {\it monoid},  we  mean a commutative
semigroup with identity which satisfies the cancellation law (that
is, if $a,b ,c$ are elements of the monoid with $ab = ac$, then $b =
c$ follows). The multiplicative semigroup of non-zero elements of an
integral domain is a monoid.
Let $H$ be a monoid. We
denote by $H^{\times}$ the group of invertible elements of $H$ and by  $\mathcal A (H)$  the   set of atoms (irreducible elements) of $H$. If $a=u_1 \cdot \ldots \cdot u_k$, where $k \in \N$ and $u_1, \ldots, u_k \in \mathcal A (H)$, then $k$ is called the length of the factorization and  $\mathsf L (a) = \{ k \in \N \mid a \ \text{has a factorization of length } \ k \} \subseteq \N$ is the set of lengths of $a$. For convenience, we set $\mathsf L (a) = \{0\}$ if $a \in H^{\times}$.  Furthermore, we denote by
\[
\mathcal L (H) = \{\mathsf L (a) \mid a \in H \} \quad \text{the  system of sets of lengths of }  H \,.
\]
Let $k \in \mathbb N$ and suppose that  $H \ne H^{\times}$. Then
\[
\mathcal U_k (H) \ = \ \bigcup_{a \in H, \, k \in \mathsf L (a)}  \mathsf L (a)
\]
is the union of all sets of lengths containing $k$. Thus, $\mathcal U_k (H)$ is the set of all $m \in \N$ such that there are atoms $u_1, \ldots, u_k, v_1, \ldots, v_m$ with $u_1 \cdot \ldots \cdot u_k = v_1 \cdot \ldots \cdot v_m$, and we define $\rho_k (H) = \sup \, \mathcal U_k (H)$.
Sets of lengths are the best investigated invariants in Factorization Theory (for an overview we refer to \cite{Ge-HK06a, C-F-G-O16}).

\smallskip
Let $G$ be an additively written  finite abelian group. By a {\it sequence} over $G$, we mean a finite sequence of terms from $G$ where repetition is allowed and the order is disregarded. As usual (see \cite{Ge-HK06a, Gr13a}), we consider sequences as elements of the free abelian monoid $\mathcal{F}(G)$  with basis $G$. A sequence $S$ over $G$ will be written in the form
$$S=g_1\cdot \ldots \cdot g_l=\prod_{g\in G}g^{\mathsf v_g(S)}\in \mathcal{F}(G),$$
and we call
\begin{itemize}
\item[] $|S|=l=\sum_{g\in G}\mathsf v_g(S)\in \mathbb{N}_0$ the \emph{length} of $S$,
\item[] $\supp(S)=\{g\in G\mid \mathsf v_g(S)>0\}\subseteq G$ the \emph{support} of $S$, and
\item[] $\sigma(S)=\sum_{i=1}^{l}g_i=\sum_{g\in G}\mathsf v_g(S)g\in G$ the \emph{sum} of $S$.
\end{itemize}
We say that $S$ is a {\it zero-sum sequence} if $\sigma (S)=0$, and clearly the set of zero-sum sequences
\[
\mathcal B(G) = \{ S  \in \mathcal F(G) \mid \sigma (S) =0\} \subset \mathcal F (G)
\]
is a submonoid of $\mathcal F (G)$, called the
 {\it monoid of zero-sum sequences} \ over \ $G$. Clearly, an element $A \in \mathcal B (G)$ is irreducible if and only if it is a minimal zero-sum sequence, and we denote by
$\mathcal A(G) := \mathcal A \big( \mathcal B (G) \big)$ the set of atoms of $\mathcal B (G)$. This set is finite, and
the {\it Davenport constant} $\mathsf D (G)$  of $G$ is the maximal length of a minimal zero-sum sequence over $G$, thus
\[
\mathsf  D (G) = \max \bigl\{ |U| \, \bigm| \; U \in \mathcal A
(G) \bigr\} \in \N \,.
\]
In other words, $\mathsf D (G)$ is the smallest integer $\ell$ such that every sequence $S$ over $G$ of length $|S|\ge \ell$ has a nontrivial zero-sum subsequence.

\smallskip
A monoid $H$ is a {\it Krull monoid} if one of the following equivalent conditions is satisfied:
\begin{enumerate}
\item[(a)] $H$ is completely integrally closed and satisfies the ascending chain condition on divisorial ideals.

\item[(b)] There is a free abelian monoid $F$ and a homomorphism $\varphi \colon H \to F$ with the following property: if $a, b \in H$ and $\varphi (a)$ divides $\varphi (b)$ in $F$, then $a$ divides $b$ in $H$.
\end{enumerate}
We refer to the monographs \cite{HK98, Ge-HK06a}    for a detailed exposition of Krull monoids and to the already mentioned paper \cite{Ge-Gr-Yu15}. We just mention that a domain $R$ is a Krull domain if and only if its monoid of nonzero elements is a Krull monoid, and for monoids of modules which are Krull we refer to \cite{Ba-Wi13a,Ba-Ge14b,Fa06a}. Property (a) easily shows that every integrally closed noetherian domain is a Krull domain. Since the embedding $\mathcal B (G) \hookrightarrow \mathcal F (G)$ satisfies Property (b), we infer that $\mathcal B (G)$ is a Krull monoid. It is easy to verify that the class group of $\mathcal B (G)$ is isomorphic to $G$ and that every class contains a prime divisor. Furthermore, $\mathcal B (G)$ plays a universal role in the study of the arithmetic of general Krull monoids.
In particular, the system of sets of lengths of a Krull monoid $H$ with class group $G$, where each class contains a prime divisor, coincides with the system of sets of lengths of $\mathcal B (G)$. We give a precise formulation of this well-known fact (for progress in this directions see \cite[Proposition 2.2]{Ge-Gr-Yu15}).

\smallskip
\begin{proposition}[\cite{Ge-HK06a}, Theorem 3.4.10] \label{2.1}
Let $H$ be a Krull monoid with class group $G$ such that every class contains a prime divisor. Then there is a transfer homomorphism $\boldsymbol \beta \colon   H \to \mathcal B (G)$ which
implies that, for every $k \in \N$,
\[
\mathcal U_k (H) = \mathcal U_k  (\mathcal  B(G))     \quad \text{and} \quad \rho_k (H) =  \rho_k  (\mathcal  B(G))   \,.
\]
\end{proposition}
Thus the invariants $\rho_k (H)$ can be studied in the monoid of zero-sum sequences $\mathcal B (G)$. As usual, we set \   $\mathcal U_k (G) = \mathcal U_k (\mathcal B (G))$ and $\rho_k (G) = \rho_k (\mathcal B (G))$.

\bigskip
\section{Proof of Theorem \ref{Main}} \label{3}

Throughout this section, let $G$ be a finite abelian group. If $|G| \le 2$, then $\mathcal B (G)$ is factorial whence $\rho_k (G) = k$ for every $k \in \N$. Clearly, $\mathsf D^* (G)=3$ if and only if $G$ is cyclic of order three or isomorphic to $C_2 \oplus C_2$. In this case,  Inequality \eqref{inequality} is an equality, and in particular Theorem \ref{Main} holds with $k^*=1$. Thus for the remainder of this section we suppose  that $\mathsf D^* (G) \ge 4$, and this implies that $|G|\ge 4$.

\smallskip
We introduce the main concept of the present paper.

\begin{definition} \label{def-nice}
Let   $A\in \mathcal B(G)$.
\begin{enumerate}
\item We say that $A$ is {\it pair-nice} (with respect to $G$) if there is a factorization $A=U_1\cdot\ldots\cdot U_{2k}$, $k\in \N$ with the following properties

 \begin{itemize}
\item[(a)] $U_1,\ldots, U_{2k}\in \mathcal A(G)$ with $|U_1|=\ldots=|U_{2k}|=\mathsf D^*(G)$;

\item[(b)] For all $i\in [1,2k]$, there is a $g_i\in \supp(U_i)$ such that  $g_1\cdot\ldots\cdot g_{2k}$ is a   product of length $2$ atoms.
 \end{itemize}

\item We say that $A$ is {\it nice} (with respect to $G$) if there is a factorization $A=U_1\cdot\ldots\cdot U_{2k+1}$, $k\in \N$ with the following properties

   \begin{itemize}
\item[(a)] $U_1,\ldots, U_{2k+1}\in \mathcal A(G)$ with $|U_1|=\ldots=|U_{2k+1}|=\mathsf D^*(G)$;

\item[(b)] For all $i\in [1,2k+1]$, there is a $g_i\in \supp(U_i)$ such that  one of the following holds:
     \begin{itemize}
     \item[(i)] $\mathsf D^*(G)$ is odd, $A(g_1\cdot\ldots\cdot g_{2k+1})^{-1}$ is a product of length $2$ atoms, and $g_1\cdot \ldots \cdot g_{2k+1} =W_0\cdot W_1\cdot\ldots\cdot W_{k-1}$, where  $|W_0|=3$, $|W_1|=\ldots=|W_{k-1}|=2$, and $W_0,\ldots,W_{k-1}\in \mathcal A(G)$.

     \item[(ii)]  $\mathsf D^*(G)$ is even and there exists a $g_{2k+2}\in \supp\left(A(g_1\cdot\ldots\cdot g_{2k+1}\right)^{-1})$ such that  $A(g_1\cdot\ldots\cdot g_{2k+2})^{-1}$  and $g_1\cdot\ldots\cdot g_{2k+2}$ are  both  products of length $2$ atoms.
\end{itemize}
\end{itemize}
\end{enumerate}
\end{definition}

\smallskip
Suppose there exists a {\it nice}  $A\in \mathcal B(G)$, and let all notation be as in the above definition. Then
\[
\{2k+1, k\mathsf D^*(G)+\lfloor\frac{\mathsf D^*(G)}{2}\rfloor\}
\subseteq \mathsf L(A) \quad \text{and hence} \quad  \rho_{2k+1}(G)\ge k\mathsf D^*(G)+\lfloor\frac{\mathsf D^*(G)}{2}\rfloor \,.
\]
Thus, up to a small calculation (which will be done in the actual proof of Theorem \ref{Main}), the assertion of the theorem follows. Therefore the main task of  the paper is to find nice elements. We do this for groups of rank two (Lemma \ref{M1}), for groups of rank three (Lemma \ref{M2}), and then we put all together in Lemma \ref{M3}. Note, if $G$ is  cyclic  of order greater than or equal to four, then there are no nice elements. Furthermore, if $A$ is nice or pair-nice, then $0 \notin \supp (A)$.

Our first lemma gathers some basic facts  which we will use without further mention.
\begin{lemma}\label{product}
Let  $E,E_1$ be {\it pair-nice} zero-sum sequences (with respect to $G$).  Suppose that $X_1,X_2,X_3\in \mathcal A(G)$ are of length $\mathsf D^*(G)$. Then
\begin{enumerate}
\item $E\cdot E_1$ is {\it pair-nice} (with respect to $G$);
\item If $\mathsf D^*(G)$ is even and $E\cdot X_1$  is a product of length $2$ atoms, then  $E\cdot X_1$ is {\it nice} (with respect to $G$);

\item If $\mathsf D^*(G)$ is odd and there exists $a_i\in\supp(X_i)$ for each $i\in [1,3]$ such that $a_1a_2a_3\in \mathcal A(G)$  and $EX_1X_2X_3(a_1a_2a_3)^{-1}$ is a product of length $2$ atoms, then $EX_1X_2X_3$ is {\it nice} (with respect to $G$).
\end{enumerate}
\end{lemma}

\begin{proof}
Since $E$ is {\it pair-nice}, we assume that $E=U_1\cdot \ldots\cdot U_{2k}$, where $k\in \mathbb{N}$ and $U_1, \dots, U_{2k}\in \mathcal A(G)$ are of length $\mathsf D^*(G)$, and there exists $g_i\in \supp(U_i)$ for each $i\in [1, 2k]$ such that  $g_1\cdot\ldots\cdot g_{2k}$ is a product of length $2$ atoms.
\begin{enumerate}
\item It is obvious by definition.

\item Since $E\cdot X_1$ and $g_1\cdot\ldots\cdot g_{2k}$ are both products of length $2$ atoms, we obtain that $E(g_1\cdot\ldots\cdot g_{2k})^{-1}X_1$ is a product of length $2$ atoms. Therefore there exist $x\in \supp( X_1)$ and $y\in \supp( E(g_1\cdot\ldots\cdot g_{2k})^{-1})$ such that $xy\in \mathcal A(G)$. It follows that $g_1\cdot\ldots\cdot g_{2k}\cdot xy$ and $EX_1(g_1\cdot\ldots\cdot g_{2k}\cdot xy)^{-1}$ are both products of length $2$ atoms which implies that  $E\cdot X_1$ is {\it nice} by $\mathsf D^*(G)$ is even.

\item Since $EX_1X_2X_3(a_1a_2a_3)^{-1}$ and $g_1\cdot\ldots\cdot g_{2k}$ are product of length $2$ atoms, we have that  $EX_1X_2X_3(g_1\cdot\ldots\cdot g_{2k}a_1a_2a_3)^{-1}$ is a product of length $2$ atoms. Moreover, $a_1a_2a_3$ is an atom implies that  $EX_1X_2X_3$ is {\it nice} by $\mathsf D^*(G)$ is odd.
\end{enumerate}
\end{proof}

\begin{lemma}\label{M1}
Let $G=C_n\oplus C_{mn}$ with $n> 1$ and $m\in \N$. Then there exist a $k^*\in \N$ and atoms $W_1, \ldots, W_{2k^*+1}\in \mathcal A(G) $ of length $\mathsf D^*(G)$ such that $W_1\cdot\ldots \cdot W_{2k^*+1}$ is {\it nice}.
\end{lemma}

\begin{proof}
Let $(e_1, e_2)$ be a basis of $G$ with $\ord(e_1)=n$ and $\ord(e_2)=mn$. Then $\mathsf D^*(G)=mn+n-1$.

Now set
\begin{align*}
U_i& = e_1^{n-1}\left((-1)^{i+1}e_2+(i+1)e_1\right)\left((-1)^{i+1}e_2-ie_1\right)\left((-1)^{i+1}e_2\right)^{mn-2}\,,\\
V_j& = e_2^{mn-1}\left((-1)^{j+1}e_1+(j+1)e_2\right)\left((-1)^{j+1}e_1-je_2\right)\left((-1)^{j+1}e_1\right)^{n-2}\,,\\
W_j& = (e_1+e_2)^{mn-1}\left((-1)^{j+1}e_1+(j+1)(e_1+e_2)\right)\left((-1)^{j+1}e_1-j(e_1+e_2)\right)\left((-1)^{j+1}e_1\right)^{n-2}\,,
\end{align*}
where $i\in [0, n-1]$ and $j\in [0, mn-1]$. Then $|U_i|=|V_j|=|W_j|=\mathsf D^*(G)$ and $U_i, V_j, W_j\in \mathcal {A}(G)$ for all $i\in [0, n-1]$, $j\in [0, mn-1]$.

\medskip
We distinguish the following three cases.

\smallskip
{\bf Case 1:} $n$ is odd and $m$ is even.

Let $n=2\alpha+1$ with $\alpha\geq 1$. Then $m\ge 2$ and $\mathsf D^*(G)$ is even.

Since $mn$ is even, let $E_2=V_0\cdot\ldots\cdot V_{mn-1}$ and hence $E_2$ is {\it pair-nice}. By calculation, we obtain that $E_2\,e_2^{-(mn-1)mn}$ is a product of length $2$ atoms.

Let $E_3=V_0\cdot\ldots\cdot V_{(m-1)n-1}W_{(m-1)n}\cdot\ldots\cdot W_{mn-1}$ and hence  $E_3$ is {\it pair-nice}. By calculation, we obtain that  $E_3\left( e_2^{(mn-1)(m-1)n}(e_1+e_2)^{(mn-1)n} \right)^{-1} $ is a product of length $2$ atoms.

Replacing the basis $(e_1,e_2)$ with $(-e_1, e_2)$, we can construct a zero-sum sequence $E_3'$ similar with $E_3$ such that $E_3'$ is {\it pair-nice} and $E_3'\left( e_2^{(mn-1)(m-1)n}(-e_1+e_2)^{(mn-1)n}\right)^{-1}$ is a product of length $2$ atoms.

Let
$X=e_2^{mn-1}(e_1+e_2)^{\alpha+1}(e_1-e_2)^\alpha$ and
 $Y=e_2^{mn-1}(e_2-e_1)(-e_1)^{n-1}$.
Then $X$ and $Y$ are atoms of length $\mathsf D^*(G)$.
Since $XY$ and $X(-Y)$ are {\it pair-nice}, we obtain that
$$E'=E_2^{2\alpha} E_3^{\alpha+1}(-E_3')^\alpha((-X)Y)^{\frac{(mn-1)n-1}{2}}((-X)(-Y))^{\frac{(mn-1)n-1}{2}} \text{ is {\it pair-nice} .}$$

 By  calculation we have that $(-X)E'$ is a product of length $2$ atoms.
It follows that  $(-X)E'$ is {\it nice} by Lemma \ref{product}.2.

\medskip
{\bf Case 2:} $n$ is odd and $m$ is odd.

Then $\mathsf D^*(G)$ is odd.
Since $n$ is odd, let $O_1=U_0\cdot \ldots\cdot U_{n-1}$ and hence by calculation,  we can obtain that $O_1\left( e_1^{n(n-1)}(-e_2)^{mn}\right)^{-1}$ is a product of length $2$ atoms.

Since $mn$ is odd, let $O_2=V_0\cdot \ldots\cdot V_{mn-1}$ and hence by calculation, $O_2\left( e_2^{(mn-1)mn}(-e_1)^n\right)^{-1}$ is  a product of length $2$ atoms.

Replacing the basis $(e_1, e_2)$ with $(-e_1, e_2)$, we can construct a zero-sum sequence $O_1'$ similarly with $O_1$ such that $O_1'\left ((-e_1)^{n(n-1)}(-e_2)^{mn}\right)^{-1}$ is a product of length $2$ atoms.
 By the constructions of $O_1$ and $O_1'$, we can obtain that $O_1O_1'$ is {\it pair-nice} and $O_1O_1'\left(-e_2\right)^{-2mn}$ is a product of length $2$ atoms.
We denote  $O_1O_1'$ by $E$.

Let $X=(-e_1)^{n-1}e_2^{mn-1}(e_2-e_1)$ and hence $X$ is an atom of length $\mathsf D^*(G)$. Therefore we let 
 $O=O_2^{n-2}O_1XE^{\frac{(mn-1)(n-2)}{2}}$ and hence $O\left((e_2-e_1)e_1(-e_2)\right)^{-1}$ is a product of length $2$ atoms.

Since $O_2U_i$ is {\it pair-nice} for each $i\in [0,n-1]$, we obtain that $O_2^{n-2}O_1(U_0U_1)^{-1}$ is {\it pair-nice} and hence $O(U_0U_1X)^{-1}$ is {\it pair-nice}.

By $(e_2-e_1)\in \supp( X)$, $-e_2\in\supp( U_0)$, and $e_1\in \supp( U_1)$, Lemma \ref{product}.3 implies  that   $O$ is {\it nice}.

\medskip
{\bf Case 3:} $n$ is even.

Then $\mathsf D^*(G)$ is odd.
Since $n$ is even, let $E_1=U_0\cdot \ldots\cdot U_{n-1}$ and hence $E_1$ is {\it pair-nice} and  $E_1e_1^{-n(n-1)}$ is a product of length $2$ atoms.
Let $E_2=V_0\cdot \ldots\cdot V_{mn-1}$ and hence $E_2$ is {\it pair-nice} and  $E_2e_2^{-(mn-1)mn} $ is a product of length $2$ atoms.

Let \begin{align*}
X_1=&e_1^{n-1}e_2^{mn-1}(e_1+e_2)\,,\\ X_2=&e_2^{mn-1}e_1(e_1+e_2)^{\frac{n}{2}}(e_1-e_2)^{\frac{n}{2}-1}\,,\\  Y=&e_1^{n-1}(-e_2)^{mn-1}(e_1-e_2)\,.
\end{align*}
 Then $X_1$, $X_2$, and $Y$ are atoms of length $\mathsf D^*(G)$. Since $X_1Y$ and $X_1(-Y)$ are both {\it pair-nice}, we obtain that
 $(X_1Y)^{\frac{mn}{2}}(X_1(-Y))^{\frac{mn}{2}}(-E_1)^m(-E_2)$ is {\it pair-nice} and
$$(X_1Y)^{\frac{mn}{2}}(X_1(-Y))^{\frac{mn}{2}}(-E_1)^m(-E_2)(e_1+e_2)^{-mn} \text{ is product of length $2$ atoms } \,.$$

Replacing the basis $(e_1, e_2)$ with $(e_1, e_2-e_1)$,
we can construct a zero-sum sequence $E$ similarly  such that $E$ is {\it pair-nice} and $Ee_2^{-mn}$ is a product of length $2$ atoms.

Then let $E'=(X_1Y)^{\frac{n}{2}}(X_1(-Y))^{\frac{n}{2}}(-E_1)(-E)^n$ and hence $E'$ is {\it pair-nice} and $$E'\left((e_1+e_2)^n(-e_2)^n\right)^{-1} \text{ is a product of length $2$ atoms }.$$
Similarly with $E'$,  if we replace the basis $(e_1, e_2)$ with $(e_1, -e_2)$, we can construct a zero-sum sequence $E''$ such that $E''$  is {\it pair-nice} and
$$E''\left((e_1-e_2)^ne_2^n\right)^{-1} \text{ is a product of length $2$ atoms }.$$
Since  $X_2Y$ and $X_2(-Y)$ are both {\it pair-nice}, we
let $E'''=(X_2Y)^{\frac{n}{2}}(X_2(-Y))^{\frac{n}{2}}(-E)^n(-E')^{\frac{n}{2}}(-E'')^{\frac{n}{2}-1}$ and hence $E'''$ is {\it pair-nice} and
$E'''e_1^{-n}$ is a product of length $2$ atoms.
It follows that $(-E)(-E''')$ is {\it pair-nice}, $(e_1+e_2)\in\supp( X_1)$, $-e_2\in \supp (Y)$, $-e_1\in \supp (-Y)$, and
$$X_1Y(-Y)(-E)(-E''') \left((-e_2)(-e_1)(e_1+e_2)\right)^{-1} \text{ is a product of length $2$ atoms }$$
which implies that  $X_1Y(-Y)(-E)(-E''')$ is {\it nice} by Lemma \ref{product}.3.
\end{proof}

\begin{lemma}\label{M2}
Let $G=C_{n_1}\oplus C_{n_2}\oplus C_{n_3}$ with $1<n_1\t n_2\t n_3$. Then there exist a $k^*\in \N$ and atoms $W_1,  \ldots, W_{2k^*+1}\in \mathcal A(G) $ of length $\mathsf D^*(G)$ such that  $W_1\cdot\ldots \cdot W_{2k^*+1}$ is {\it nice}.
\end{lemma}

\begin{proof}
Let $(e_1, e_2, e_3)$ be a basis of $G$ with $\ord(e_1)=n_1$, $\ord(e_2)=n_2$, and $\ord(e_3)=n_3$. Then $\mathsf D^*(G)=n_1+n_2+n_3-2$.
Denote
\begin{align*}
X_1 & =e_1^{n_1-1}e_2^{n_2-1}(-e_3)^{n_3-2}(e_1-e_3)(e_2-e_3),\\
X_2 & =e_1^{n_1-1}(-e_2)^{n_2-2}e_3^{n_3-1}(e_1-e_2)(e_3-e_2),\\
X_3 & =(-e_1)^{n_1-2}e_2^{n_2-1}e_3^{n_3-1}(-e_1+e_2)(-e_1+e_3).
\end{align*}
It is easy to see that $X_i$ is an atom of length $\mathsf 	D^*(G)$ for each $i\in [1, 3]$. Thus
$$X_1X_2X_3 \left(e_1^{n_1}e_2^{n_2}e_3^{n_3}\right)^{-1} \text{ is a product of length $2$ atoms }.$$
If $n_1=n_2=n_3=2$, we have $X_1X_2X_3$ is a product of length $2$ atoms and hence $X_1X_2X_3$ is {\it nice}.
Thus we can assume $n_3\geq 4$.

Denote
\begin{align*}
X_1' & =e_1^{n_1-1}(-e_2)^{n_2-1}e_3^{n_3-2}(e_1+e_3)(-e_2+e_3),\\
X_2' & =e_1^{n_1-1}e_2^{n_2-2}(-e_3)^{n_3-1}(e_1+e_2)(-e_3+e_2),\\
X_3' & =(-e_1)^{n_1-2}(-e_2)^{n_2-1}(-e_3)^{n_3-1}(-e_1-e_2)(-e_1-e_3).
\end{align*}
 Thus $E_1=X_1X_1'X_2X_2'X_3X_3'$ is {\it pair-nice} and $E_1 e_1^{-2n_1}$ is a product of length $2$ atoms. Similarly, we can construct {\it pair-nice} zero-sum sequences $E_2$ and $E_3$ such that $E_2 e_2^{-2n_2}$ and $E_3e_3^{-2n_3}$ are both products of length $2$ atoms.

Set
\begin{align*}
U_i=& e_1^{n_1-1}\left((-1)^{i+1}e_3+(i+1)e_1\right)\left((-1)^{i+1}e_3-ie_1\right)\left((-1)^{i+1}e_3\right)^{n_3-3}
\left((-1)^{i+1}(e_2+e_3)\right)\left((-1)^{i+1}e_2\right)^{n_2-1}\,,\\
V_j=& e_2^{n_2-1}\left((-1)^{j+1}e_3+(j+1)e_2\right)\left((-1)^{j+1}e_3-je_2\right)\left((-1)^{j+1}e_3\right)^{n_3-3}
\left((-1)^{j+1}(e_1+e_3)\right)\left((-1)^{j+1}e_1\right)^{n_1-1}\,,\\
W_l=&\left\{\begin{aligned}
&e_3^{n_3-1}\left((-1)^{l+1}e_2+(l+1)e_3\right)\left((-1)^{l+1}e_2-le_3\right)\left((-1)^{l+1}e_2\right)^{n_2-3}
\left((-1)^{l+1}(e_1+e_2)\right)\left((-1)^{l+1}e_1\right)^{n_1-1}\,,\\
&\hspace{11cm} \qquad\,\text{ if $n_2\ge 3$ ,}\\
&e_3^{n_3-1}\left(e_1+le_2+(l+1)e_3\right)\left(e_1+(l+1)e_2-le_3\right)e_2\,, \hspace{3.5cm}\quad\text{ if  $n_1=n_2=2$ , }
\end{aligned}\right.
\end{align*}
where $i\in [0, n_1-1]$, $j\in [0, n_2-1]$, and $l\in [0, n_3-1]$. It is easy to see that, $V_i$, $U_j$, $W_l$ are all atoms of length $\mathsf D^*(G)$, where $i\in [0, n_1-1]$, $j\in [0, n_2-1]$, and $l\in [0, n_3-1]$.

\medskip
Now we distinguish the following four cases.

\smallskip
\noindent
{\bf Case 1:} $n_1$ is even.

Then $\mathsf D^*(G)$ is even since $1<n_1\mid n_2\mid n_3$.

Since $n_1$ is even, we let $E_1'=U_0\cdot\ldots\cdot U_{n_1-1}$,  $E_2'=V_0\cdot\ldots\cdot V_{n_2-1}$, and $E_3'=W_0\cdot\ldots\cdot W_{n_3-1}$ and hence $E_1'$, $E_2'$, and $E_3'$ are {\it pair-nice}. Moreover $E_1'e_1^{-n_1(n_1-1)}$, $E_2' e_2^{-n_2(n_2-1)}$, and $E_3' e_3^{-n_3(n_3-1)}$ are products of length $2$ atoms.

Therefore let $E_4=E_1'(-E_1)^{\frac{n_1}{2}-1}$, $E_5=E_2'(-E_2)^{\frac{n_2}{2}-1}$, and $E_6=E_3'(-E_3)^{\frac{n_3}{2}-1}$,
and hence $E_4,E_5,E_6$ are {\it pair-nice} and $E_4e_1^{-n_1}$, $E_5 e_2^{-n_2}$, and $E_6 e_3^{-n_3}$ are products of length $2$ atoms.

It follows that $O=E_4E_5E_6(-X_1)(-X_2)(-X_3)$ is a product of length $2$ atoms and $E_4E_5E_6(-X_1)(-X_2)$  is {\it pair-nice}. Then $O$ is {\it nice} by Lemma \ref{product}.2.

\smallskip
\noindent
{\bf Case 2:} $n_1$ is odd, $n_2$ is even.

Then $\mathsf D^*(G)$ is odd.
Since $n_1$ is odd, we let $O_1=U_0\cdot\ldots\cdot U_{n_1-1}$ and hence $O_1U_0^{-1}$ is {\it pair-nice} and $$O_1 \left(e_1^{(n_1-1)n_1}(-e_3)^{n_3-1}(-e_2-e_3)(-e_2)^{n_2-1}\right)^{-1} \text{ is a product of length $2$ atoms }\,.$$

Since $n_2$ is even, we let  $E_2'=V_0\cdot\ldots\cdot V_{n_2-1}$ and $E_3'=W_0\cdot\ldots\cdot W_{n_3-1}$ and hence $E_2'$ and $E_3'$ are {\it pair-nice}. Moreover $E_2'e_2^{-n_2(n_2-1)}$ and $E_3' e_3^{-n_3(n_3-1)}$ are both products of length $2$ atoms.

Therefore let  $E_5=E_2'(-E_2)^{\frac{n_2}{2}-1}$ and $E_6=E_3'(-E_3)^{\frac{n_3}{2}-1}$,
and hence $E_5,E_6$ are {\it pair-nice} and   $E_5 e_2^{-n_2}$ and $E_6 e_3^{-n_3}$ are both  products of length $2$ atoms.

Let $O=(-E_1)^{\frac{n_1-1}{2}}E_5E_6O_1X_1(-X_1)$ and hence $O \left(e_2e_3(-e_2-e_3)\right)^{-1}$ is a product of length $2$ atoms.

Since $(-e_2-e_3)\in \supp (U_0)$, $e_2\in\supp (X_1)$, $e_3\in\supp (-X_1)$, and $O\left(U_0X_1(-X_1)\right)^{-1}$ is {\it pair-nice}, we obtain that  $O$ is {\it nice} by Lemma \ref{product}.3.

\smallskip
\noindent
{\bf Case 3:} $n_1$ is odd, $n_2$ is odd, and $n_3$ is even.

Then $\mathsf D^*(G)$ is even.
Since $n_1$ is odd, we let
 $O_1=U_0\cdot\ldots\cdot U_{n_1-1}$  and hence $O_1U_0^{-1}$ is {\it pair-nice} and

$$O_1 \left(e_1^{(n_1-1)n_1}(-e_3)^{n_3-1}(-e_2-e_3)(-e_2)^{n_2-1}\right)^{-1} \text{ is a product of length $2$ atoms }\,.$$

  Since $n_3$ is even, we let $E_3'=W_0\cdot\ldots\cdot W_{n_3-1}$ and hence $E_3'$ is {\it pair-nice} and $E_3' e_3^{-(n_3-1)n_3}$ is a product of length $2$ atoms. Therefore $E_6=E_3'(-E_3)^{\frac{n_3}{2}-1}$ is {\it pair-nice} and $E_6 e_3^{-n_3}$ is a product of length $2$ atoms.

Since $n_2\ge 3$, we let
\begin{align*}
W_l'=& (e_2+e_3)^{n_3-1}\left((-1)^{l+1}e_2+(l+1)(e_2+e_3)\right)\left((-1)^{l+1}e_2-l(e_2+e_3)\right)\left((-1)^{l+1}e_2\right)^{n_2-3}\cdot\\
&\left((-1)^{l+1}(e_1+e_2)\right)\left((-1)^{l+1}e_1\right)^{n_1-1}\,,
\end{align*}
where $l\in [0,n_3-1]$. Thus $W_l'$ is an atom of length $\mathsf D^*(G)$ for each $l\in [0,n_3-1]$.

Let $E_7=(W_0W_1)\cdot\ldots\cdot (W_{n_3-n_2-1}W_{n_3-n_2}')\cdot\ldots\cdot (W_{n_3-2}'W_{n_3-1}')$ and hence $E_7$ is {\it pair-nice} and $$E_7 \left(e_3^{(n_3-1)(n_3-n_2)}(e_2+e_3)^{(n_3-1)n_2}\right)^{-1} \text{ is a product of length $2$ atoms }.$$

Therefore   let $O'=O_1^{(n_3-1)n_2}E_7(-E_1)^{\frac{(n_3-1)n_2(n_1-1)n_1}{2n_1}}E_2^{\frac{(n_3-1)n_2(n_2-1)}{2n_2}}E_6^{\frac{(n_3-1)n_3(n_2-1)}{n_3}}$ and hence $O'$ is a product of length $2$ atoms  and $O'(U_0)^{-(n_3-1)n_2}$ is {\it pair-nice}.

Let
$Y=e_1^{n_1-1}e_2^{n_2-1}e_3^{n_3-1}(e_1+e_2+e_3)
$ and hence $Y$ is an atom of length $\mathsf D^*(G)$. Since $YU_0$ and $(-Y)U_0$ are both {\it pair-nice}, we obtain that
$U_0^{(n_3-1)n_2-1}(Y(-Y))^\frac{(n_3-1)n_2-1}{2}=(U_0Y)^{\frac{(n_3-1)n_2-1}{2}}(U_0(-Y))^{\frac{(n_3-1)n_2-1}{2}}$ is {\it pair-nice}.

It follows that $O=O'(Y(-Y))^\frac{(n_3-1)n_2-1}{2}$ is a product of length $2$ atoms and
$$OU_0^{-1}=O'(U_0)^{-(n_3-1)n_2}\cdot U_0^{(n_3-1)n_2-1}(Y(-Y))^\frac{(n_3-1)n_2-1}{2} \text{ is {\it pair-nice}.} $$

Then $O$ is {\it nice} by Lemma \ref{product}.2.

\smallskip
\noindent
{\bf Case 4:} $n_1$ is odd, $n_2$ is odd, and $n_3$ is odd.

Then $\mathsf D^*(G)$ is odd.
Since $n_3$ is odd, we let
 $O_3=W_0\cdot\ldots\cdot W_{n_3-1}$  and hence $O_3W_0^{-1}$ is {\it pair-nice} and
 $$
O_3 \left(e_3^{(n_3-1)n_3}(-e_2)^{n_2-1}(-e_1-e_2)(-e_1)^{n_1-1}\right)^{-1} \text{ is a product of length $2$ atoms }\,,
  $$
  and hence
   $$O_3(-E_3)^{\frac{n_3-1}{2}}X_1X_2X_3 \big(e_3^{n_3}e_1e_2(-e_1-e_2)\big)^{-1}\text{ is  a product of length $2$ atoms .}$$
    By $(W_0X_1)(X_2X_3)$ is {\it pair-nice}, we have that $O_3(-E_3)^{\frac{n_3-1}{2}}X_1X_2X_3$ is {\it pair-nice}.

Similarly, if we replace the basis $(e_1,e_2,e_3)$ with $(-e_1,-e_2,e_1+e_2+e_3)$, we can construct a zero-sum sequence $E$ such that $E$ is {\it pair-nice} and $$E \left((e_1+e_2+e_3)^{n_3}(-e_1)(-e_2)(e_1+e_2)\right)^{-1} \text{ is a product of length $2 $ atoms .}$$

Let \begin{align*}
Y=&e_1^{n_1-1}e_2^{n_2-1}e_3^{n_3-1}(e_1+e_2+e_3)\,,\\
Y_1=&e_1^{n_1-1}(-e_1-e_2)^{n_2-1}(-e_3)^{n_3-1}(-e_2-e_3)\,,
\end{align*}
and hence $Y$ and $Y_1$ are  atoms of length $\mathsf D^*(G)$.

 Therefore  let \begin{align*}
 O_4=&Y^{n_3}(-E_1)^{\frac{(n_1-1)n_3}{2n_1}}(-E_2)^{\frac{(n_2-1)n_3}{2n_2}}(-E_3)^{\frac{(n_3-1)}{2}}\,,\\
 O=&O_4(Y_1(-Y_1))^{\frac{n_3-1}{2}}(-E)\,.
 \end{align*}
 Then $O_4 (e_1+e_2+e_3)^{-n_3}$ is a product of length $2$ atoms and hence
 $O \left(e_1e_2(-e_1-e_2)\right)^{-1}$ is a product of length $2$ atoms.
  By calculation, we obtain that
$$O(Y^2Y_1)^{-1}=O_4(Y)^{-n_3}\cdot (YY_1)^{\frac{n_3-3}{2}}\cdot (Y(-Y_1))^{\frac{n_3-1}{2}}\cdot (-E) \text{ is {\it pair-nice,}}$$
Therefore $e_1\in \supp( Y)$, $e_2\in\supp (Y)$, and $-e_1-e_2\in\supp (Y_1)$
 imply that $O$ is {\it nice} by Lemma \ref{product}.3.

\end{proof}

\begin{lemma}\label{M3}
Let  $G=G_1\oplus G_2$, where $G_1, G_2 \subset G$ are noncyclic subgroups of $G$ satisfying $\mathsf r(G)=\mathsf r(G_1)+\mathsf r(G_2)$. Suppose that  there exist $k \in \N$ and  atoms $U_1,  \ldots, U_{2k+1}\in \mathcal A(G_1) $ of length $\mathsf D^*(G_1)$ and atoms $V_1,  \ldots, V_{2k+1}\in \mathcal A(G_2) $ of length $\mathsf D^*(G_2)$ such  that  $U_1\cdot\ldots\cdot U_{2k+1}$ is {\it nice} (with respect to $G_1$) and $V_1\cdot\ldots\cdot V_{2k+1}$ is  {\it nice} (with respect to $G_2$).   Then there exist atoms $W_1,  \ldots, W_{2k+1}\in \mathcal A(G) $ of length $\mathsf D^*(G)$ such that   $W_1\cdot\ldots \cdot W_{2k+1}$ is {\it nice} (with respect to $G$).

\end{lemma}
\begin{proof}

Without loss of generality, we  can distinguish the following three cases.

\medskip
{\bf Case 1.} $\mathsf D^*(G_1)$ and $\mathsf D^*(G_2)$ are odd.

Since $U_1\cdot\ldots\cdot U_{2k+1}$ is {\it nice} (with respect to $G_1$), without loss of generality, we can assume that there exist $g_i\in \supp(U_i)$ for each $i\in [1,2k+1]$ such that $\sigma(g_1g_2g_3)=0$, $g_{2j}=-g_{2j+1}$ for each $j\in [2,k]$, and $U_1\cdot\ldots\cdot U_{2k+1}(g_1\cdot\ldots\cdot g_{2k+1})^{-1}$ is a product of  length $2$ atoms.

With the same reason,  we can assume that there exist $h_i\in \supp(V_i)$ for each $i\in [1,2k+1]$ such that $\sigma(h_1h_2h_3)=0$, $h_{2j}=-h_{2j+1}$ for each $j\in [2,k]$, and $V_1\cdot\ldots\cdot V_{2k+1}(h_1\cdot\ldots\cdot h_{2k+1})^{-1}$ is a product of  length $2$ atoms.

Let $W_i=U_ig_i^{-1}\cdot V_ih_i^{-1}\cdot (g_i+h_i)$ for all $i\in [1,2k+1]$. Then $W_i$ is an atom over $G$ of length $\mathsf D^*(G_1)+\mathsf D^*(G_2)-1=\mathsf D^*(G)$, and
$$(g_1+h_1)\cdot \ldots \cdot (g_{2k+1}+h_{2k+1})=Z_0\cdot Z_1\cdot\ldots\cdot Z_{k-1}$$
where $Z_0=(g_1+h_1)(g_2+h_2)(g_3+h_3)\in \mathcal A(G)$ and  $Z_i=(g_{2i+2}+h_{2i+2})(g_{2i+3}+h_{2i+3})\in \mathcal A(G)$ for each $i\in [1,k-1]$,
$$W_1\cdot\ldots\cdot W_{2k+1}\big((g_1+h_1)\cdot \ldots \cdot (g_{2k+1}+h_{2k+1})\big)^{-1}=U_1\cdot\ldots\cdot U_{2k+1}(g_1\cdot\ldots\cdot g_{2k+1})^{-1}\cdot V_1\cdot\ldots\cdot V_{2k+1}(h_1\cdot\ldots\cdot h_{2k+1})^{-1}$$
is a product of atoms of length $2$.

It follows that  $W_1\cdot\ldots\cdot W_{2k+1}$ is {\it nice} (with respect to $G$) by $\mathsf D^*(G)$ is odd.

\medskip
{\bf Case 2.} $\mathsf D^*(G_1)$ and $\mathsf D^*(G_2)$ are even.

Since $U_1\cdot\ldots\cdot U_{2k+1}$ is {\it nice} (with respect to $G_1$), without loss of generality, we can assume that there exist $g_i\in\supp(U_i)$  for each $i\in [1,2k+1]$ and $g_{2k+2}\in\supp(U_1g_1^{-1})$  such that $g_1=-g_2$, $g_{2k+2}=-g_3$, and $g_{2j}=-g_{2j+1}$ for each $j\in [2,k]$ and $U_1\cdot\ldots\cdot U_{2k+1}(g_1\cdot\ldots\cdot g_{2k+2})^{-1}$ is a product of  length $2$ atoms.

With the same reason,  we can assume that there exist $h_i\in \supp(V_i)$   for each $i\in [1,2k+1]$ and $h_{2k+2}\in\supp(V_1h_1^{-1})$  such that $h_1=-h_2$, $h_{2k+2}=-h_3$,  and $h_{2j}=-h_{2j+1}$ for each $j\in [2,k]$ and $V_1\cdot\ldots\cdot V_{2k+1}(h_1\cdot\ldots\cdot h_{2k+2})^{-1}$ is a product of atoms of length $2$.

Let $W_i=U_ig_i^{-1}\cdot V_ih_i^{-1}\cdot (g_i+h_i)$ for all $i\in [4,2k+1]$ and
\begin{align*}
W_1&=U_1(g_1g_{2k+2})^{-1}\cdot V_2h_2^{-1}\cdot (g_1+h_2)\cdot g_{2k+2}\,,\\
W_2&=U_2g_2^{-1}\cdot V_1(h_1h_{2k+2})^{-1}\cdot (g_2+h_1)\cdot h_{2k+2}\,,\\
W_3&=U_3g_3^{-1}\cdot V_3h_3^{-1}\cdot (g_3+h_3)\,.
\end{align*}

 Then $W_i$ is an atom over $G$ of length $\mathsf D^*(G_1)+\mathsf D^*(G_2)-1=\mathsf D^*(G)$ for all $i\in [1,2k+1]$. It follows that
$$g_{2k+2}\cdot h_{2k+2} \cdot (g_3+h_3)\cdot (g_4+h_4)  \ldots \cdot (g_{2k+1}+h_{2k+1})=Z_0\cdot Z_1\cdot\ldots\cdot Z_{k-1}\,,$$  where $Z_0=g_{2k+2}h_{2k+2}(g_3+h_3)\in \mathcal A(G)$ and  $Z_i=(g_{2i+2}+h_{2i+2})(g_{2i+3}+h_{2i+3})\in \mathcal A(G)$ for each $i\in [1,k-1]$.

Moreover
\begin{align*}
&W_1\cdot\ldots\cdot W_{2k+1}\big(g_{2k+2}\cdot h_{2k+2} \cdot (g_3+h_3)\cdot (g_4+h_4)\cdot  \ldots \cdot (g_{2k+1}+h_{2k+1})\big)^{-1}\\
=&U_1\cdot\ldots\cdot U_{2k+1}(g_1\cdot\ldots\cdot g_{2k+2})^{-1}\cdot V_1\cdot\ldots\cdot V_{2k+1}(h_1\cdot\ldots\cdot h_{2k+2})^{-1}\cdot (g_1+h_2)(g_2+h_1)
\end{align*}
is a product of atoms of length $2$.

Thus $W_1\cdot\ldots\cdot W_{2k+1}$ is {\it nice} (with respect to $G$) by $\mathsf D^*(G)$ is odd.

\medskip
{\bf Case 3.} $\mathsf D^*(G_1)$ is even and $\mathsf D^*(G_2)$ is odd.

Since $U_1\cdot\ldots\cdot U_{2k+1}$ is {\it nice} (with respect to $G_1$) and $\mathsf D^*(G_1)$ is even, without loss of generality, we can assume that there exist $g_i\in\supp(U_i)$  for each $i\in [1,2k+1]$ and $g_{2k+2}\in \supp(U_1g_1^{-1})$ such that $g_1=-g_2$, $g_{2k+2}=-g_3$, and $g_{2j}=-g_{2j+1}$ for each $j\in [2,k]$ and $U_1\cdot\ldots\cdot U_{2k+1}(g_1\cdot\ldots\cdot g_{2k+2})^{-1}$ is a product of  length $2$ atoms.

Since $V_1\cdot\ldots\cdot V_{2k+1}$ is {\it nice} (with respect to $G_2$) and $\mathsf D^*(G_2)$ is odd, without loss of generality,  we can assume that there exist $h_i\in \supp(V_i)$ for each $i\in [1,2k+1]$ such that $\sigma(h_1h_2h_3)=0$ and $h_{2j}=-h_{2j+1}$ for each $j\in [2,k]$ and $V_1\cdot\ldots\cdot V_{2k+1}(h_1\cdot\ldots\cdot h_{2k+1})^{-1}$ is a product of  length $2$ atoms.

Let $W_i=U_ig_i^{-1}\cdot V_ih_i^{-1}\cdot (g_i+h_i)$ for all $i\in [4,2k+1]$ and
\begin{align*}
W_1&=U_1(g_1g_{2k+2})^{-1}\cdot V_1h_1^{-1}\cdot (g_1-h_2)\cdot (g_{2k+2}+h_1+h_2)\,,\\
W_2&=U_2g_2^{-1}\cdot V_2h_2^{-1}\cdot (g_2+h_2)\,,\\
W_3&=U_3g_3^{-1}\cdot V_3h_3^{-1}\cdot (g_3+h_3)\,.
\end{align*}

 Then $W_i$ is an atom over $G$ of length $\mathsf D^*(G_1)+\mathsf D^*(G_2)-1=\mathsf D^*(G)$ for all $i\in [1,2k+1]$. It follows that
$$(g_1-h_2)\cdot (g_{2k+2}+h_1+h_2)\cdot (g_2+h_2) \cdot (g_3+h_3)\cdot (g_4+h_4) \cdot \ldots \cdot (g_{2k+1}+h_{2k+1})$$ is a product of  length $2$ atoms and
\begin{align*}
&W_1\cdot\ldots\cdot W_{2k+1}\Big((g_1-h_2)\cdot (g_{2k+2}+h_1+h_2)\cdot (g_2+h_2) \cdot (g_3+h_3)\cdot (g_4+h_4) \cdot  \ldots \cdot (g_{2k+1}+h_{2k+1})\Big)^{-1}\\
=&U_1\cdot\ldots\cdot U_{2k+1}(g_1\cdot\ldots\cdot g_{2k+2})^{-1}\cdot V_1\cdot\ldots\cdot V_{2k+1}(h_1\cdot\ldots\cdot h_{2k+1})^{-1}
\end{align*}
is a product of  length $2$ atoms.

Thus $W_1\cdot\ldots\cdot W_{2k+1}$ is {\it nice} (with respect to $G$) by $\mathsf D^*(G)$ is even.

\end{proof}

\begin{proof}[\bf Proof of  Theorem \ref{Main}]
Let $H$ be a Krull monoid with finite noncyclic class group $G$ such that every class contains a prime divisor. By Proposition \ref{2.1} we have $\rho_k (H) = \rho_k (G)$ for every $k \in \N$.
If $G \cong C_2 \oplus C_2$, then $\mathsf D^* (G)=3$ and the assertion of the theorem follows from Inequality \ref{inequality} with $k^*=1$. Suppose that $\mathsf D^* (G) \ge 4$. We start with the following assertion.

\noindent{\bf Assertion.} There exist a $k^*\in \N$ and atoms $W_1, \ldots, W_{2k^*+1}$ over $G$ of length $\mathsf D^*(G)$ such that $W_1\cdot\ldots\cdot W_{2k^*+1}$ is {\it nice} (with respect to $G$).
\medskip

\noindent{\it Proof of Assertion. } We proceed by induction on $\mathsf r(G)$.

If $\mathsf r(G)=2$ or $3$, then the Assertion follows by Lemma \ref{M1} and \ref{M2}. Assume that $\mathsf r (G) \ge 4$ and suppose that the Assertion is true for all groups of smaller rank. Let $G=G_1\oplus G_2$ with $\mathsf r(G_1)=r-2$ and $\mathsf r(G_2)=2$. Then by our assumption, there exist a $k_1\in \N$ and atoms $U_1,\ldots ,U_{2k_1+1}$ over $G_1$ of length $\mathsf D^*(G_1)$ such that $U_1\cdot\ldots\cdot U_{2k_1+1}$ is {\it nice} (with respect to $G_1$).
By Lemma \ref{M1}, there exist a $k_2\in \N$ and  atoms $V_1,\ldots ,V_{2k_2+1}$ over $G_2$ of length $\mathsf D^*(G_2)$ such that $V_1\cdot\ldots\cdot V_{2k_2+1}$ is {\it nice} (with respect to $G_2$).

Let  $k^*=\max(k_1,k_2)$. Without loss of generality, we can assume that $k_1=k^*\ge k_2$. Thus $k_1-k_2$ is even and hence $V_1\cdot\ldots\cdot V_{2k_2+1}\cdot (V_1(-V_1))^{k_1-k_2}$ is {\it nice} (with respect to $G_2$). Therefore the Assertion follows by Lemma \ref{M3}.
\qed{(Proof of Assertion)}

\medskip
By the very definition of nice elements (and outlined in detail after Definition \ref{def-nice}), it follows that
\begin{equation}\label{E1}
\rho_{2k^*+1}(G)\ge  k^*\mathsf D^*(G)+\left\lfloor \frac{\mathsf D^*(G)}{2}\right\rfloor\,.
\end{equation}
Let $k \ge k^*$. Since $\rho_{2(k-k^*)} (G) = (k-k^*)\mathsf D (G)$ and $\mathcal U_{2(k-k^*)} (G) + \mathcal U_{2k^*+1}(G) \subseteq \mathcal U_{2k+1}(G)$, it follows that
\[
\rho_{2k+1}(G) \ge \rho_{2(k-k^*)}(G) + \rho_{2k^*+1}(G) \ge (k-k^*)\mathsf D (G) +  k^*\mathsf D^*(G)+\left\lfloor \frac{\mathsf D^*(G)}{2}\right\rfloor\,.
\]
If $\mathsf D(G)=\mathsf D^*(G)$, then the assertion follows from Inequality \ref{inequality}.
\end{proof}

\bibliographystyle{amsplain}

\end{document}